\documentclass[a4paper,12pt]{amsart}

\usepackage{amsmath,amsthm,amsfonts,amssymb,mathrsfs,setspace,graphicx}
\usepackage{float}
\usepackage[margin=1.2in]{geometry}
\theoremstyle{plain}
\newtheorem{theorem}{Theorem}[section]

\newtheorem{corollary}[theorem]{Corollary}

\theoremstyle{definition}
\newtheorem{example}[theorem]{Example}

\theoremstyle{definition}
\newtheorem{definition}[theorem]{Definition}

\theoremstyle{remark}
\newtheorem{remark}[theorem]{Remark}
\numberwithin{equation}{section}

\allowdisplaybreaks

\begin{document}
\title{{Invariants of Bipartite Kneser B type-\MakeLowercase{k} graphs}}

\author{Jayakumar C.}
\address{Department of  Mathematics, University College, Thiruvananthapuram, India.}
\email{jkcwarrier@gmail.com}
\author{Sreekumar K.G.}
\address{Department of  Mathematics,   University  of Kerala,  Thiruvananthapuram,    India.}
\email{sreekumar3121@gmail.com}
\author{Manilal K.}
\address{Department of Mathematics, University College, Thiruvananthapuram, India.}
\email{manilalvarkala@gmail.com}

\author{Ismail Naci Cangul}
\address{Department of Mathematics, Bursa Uludag University, Bursa 16059 Bursa-Turkey.}
\email{ cangul@uludag.edu.tr}
	
	
\begin{abstract}
Let $\mathscr{B}_n = \{ \pm x_1, \pm x_2, \pm x_3, \cdots, \pm x_{n-1}, x_n \}$ where $n>1$ is fixed, $x_i \in \mathbb{R}^+$, $i = 1, 2, 3, \cdots, n$ and $x_1 < x_2 < x_3 < \cdots < x_n$. Let $\phi(\mathscr{B}_n)$ be the set of all non-empty subsets $S = \{u_1, u_2,\cdots, u_t\}$ of $\mathscr{B}_n$ such that $|u_1|<|u_2|<\cdots <|u_{t-1}|<u_t $ where $u_t\in \mathbb{R}^+$. Let $\mathscr{B}_n^+ = \{ x_1, x_2, x_3, \cdots, x_{n-1}, x_n \}$. For a fixed $k$, let $V_1$ be the set of $k$-element subsets of $\mathscr{B}_n^+$, $1 \leq k <n$. $V_2= \phi(\mathscr{B}_n)-V_1$. For any $A \in V_2$, let $A^\dag = \{\lvert x \rvert: x \in A\}$. Define a bipartite graph with parts $V_1$ and $V_2$ and having adjacency as $X \in V_1$ is adjacent to $Y\in V_2$ if and only if $X \subset Y^\dag$ or $Y^\dag \subset X$. A graph of this type is called a bipartite Kneser B type-$k$ graph and denoted by $H_B(n,k)$. In this paper, we calculated various graph invariants of $H_B(n,k)$.
\end{abstract}
	\maketitle 
	
	\onehalfspacing
\section{Introduction}
Named after the German mathematician Martin Kneser, Kneser graphs are an interesting family of combinatorial structures in the field of graph theory. These graphs have applications in several fields, such as algebraic geometry, combinatorics, and topology.
	
Numerous fields, such as combinatorics, topology, coding theory, and combinatorial optimisation have Kneser graph applications. These are fundamental building blocks of combinatorial theory and can lead to interesting problems and conjectures. In this subject, questions about their chromatic number \cite{chromatickneser} and other graph-theoretic properties continue to be crucial to research. Kneser graphs are related to topological problems, for example, by helping to understand the homotopy type of some spaces\cite{homotopykneser}. Kneser graphs are used in coding theory\cite{knesercode} to design codes with efficient error-correcting features. 
 
The Kneser graph $K(n,k)$ has the $k$-subsets of $[n]$ as vertices. If the $k$-subsets are disjoint, then two vertices are adjacent. For integers $k\ge 1 $ and $n\ge  2k+1$, any vertex in the bipartite Kneser graph $H(n,k)$ is either a $k$-element subset or an $n-k$ element subset of  $[n]$. Here, the sets $A$ and $B$ are adjacent if  $A \subseteq B$.

The algebraic structures of different varieties of bipartite Kneser graphs, \cite{agong2018girth}, \cite{mirafzal2019automorphism}, were then built and investigated.\\

Here is a modified version of the bipartite Kneser graph : Let $\mathscr{S}_n=\{ 1, 2, 3, \cdots, n \}$ for a fixed integer $n>1$. Let $\phi(\mathscr{S}_n)$ be the set of all non-empty subsets of $\mathscr{S}_n$.  Let $V_1$ be the set of 1-element subsets of $\mathscr{S}_n$ and $V_2= \phi(\mathscr{S}_n) - V_1$. Define a bipartite graph with an adjacency of vertices as described: A vertex $A \in V_1$ is adjacent to a vertex $B \in V_2$ if and only if $A \subset B$. This graph is called a bipartite Kneser type-1 graph\cite{sreekumar2022automorphism},  and is denoted by $H_{T}(n, 1)$. Sreekumar K. G. et al., \cite{srk2023}, defined a bipartite Kneser B type-$k$ graph, $G=H_B(n,k)$, which are more general bipartite graphs.

This paper determines the following invariants of the bipartite Kneser B type-$k$ graph $H_B(n,k)$: Order, size, independence number, covering number, domination number, vertex connectivity, edge connectivity, girth, circuit rank, distance between two vertices, eccentricity, periphery, centre, median, and degree sequence.

\section{Preliminaries}

The greatest distance between any two vertices in a graph is known as its diameter. The eccentricity of a vertex, $e(v)$, is the largest possible distance between it and any other vertex.
The maximum eccentricity obtained by the vertices of a connected simple graph $G$ is the diameter, $diam(G)$. The least eccentricity among all vertices of $G$ is the radius, $rad(G)$. The centre $C(G)$ of a graph $G$ is the subgraph induced by the set of vertices with the lowest eccentricity. The periphery of $G$ is $P(G)= \{v\in V : e(v)=diam(G)\}$. The length of the shortest cycle in a graph is its girth.For any vertex $v$ of a connected graph $G$, the status of $v$ denoted as $s(v)$ is the sum of the distances from $v$ to other vertices of $G$. That is, $s(v)=\sum\limits_ {u\in V(G)} d(v,u)$. The set of vertices with minimal status is the median $M(G)$ of the graph.

The girth of a graph is the length of the shortest cycle in it.

\section{Basic definitions and examples}

\begin{definition}
Let $\mathscr{B}_n=\{ \pm x_1, \pm x_2, \pm x_3, \dots , \pm x_{n-1}, x_n \}$ where $n>1$ is fixed, $x_i \in \mathbb{R}^+$, $i= 1,2, 3, \dots, n$ and $x_1 < x_2 < x_3 < \dots < x_n$. Let $\phi(\mathscr{B}_n)$ be the set of all non-empty subsets $S=\{u_1,u_2,\dotsc,u_t\}$ of $\mathscr{B}_n$ such that $|u_1|<|u_2|<\dotsc <|u_{t-1}|<u_t $ where $u_t\in \mathbb{R}^+$. Let $\mathscr{B}_n^+=\{ x_1,  x_2,  x_3, \dots,  x_{n-1}, x_n \}$. For a fixed $k$, let  $V_1$ be the set of $k$-element subsets of $\mathscr{B}_n^+$, $1 \leq k <n$. $V_2= \phi(\mathscr{B}_n)-V_1$. For any $A\in V_2 $, let $A^\dag=\{\lvert x \rvert: x \in A\}$.
Define a bipartite  graph with parts $V_1$ and $V_2$ and having adjacency as $X \in V_1$ is adjacent to $Y \in V_2$ if and only if $X \subset Y^\dag$ or $Y^\dag \subset X$. A graph of this type is called the bipartite Kneser B type-$k$ graph \cite{srk2023} and is denoted by $H_B(n,k)$.
\end{definition}
\begin{definition}
An \textbf{$r$-vertex} in $H_B(n,k)$ is an element in $\phi(\mathscr{B}_n)$ containing $r$ elements, where $1\le r\le n$. Members of $\phi(\mathscr{B}_n)$ are called \textbf{$r$-vertices}.
\end{definition}

\begin{example}
	An example of a  bipartite  Kneser B type-1 graph, for $n=2$,namely  $H_B(2,1)$ is  shown in  \textsc{figure} \ref{first}. 

\begin{center}
	\begin{figure}[h!]	
  \includegraphics[scale=0.5]{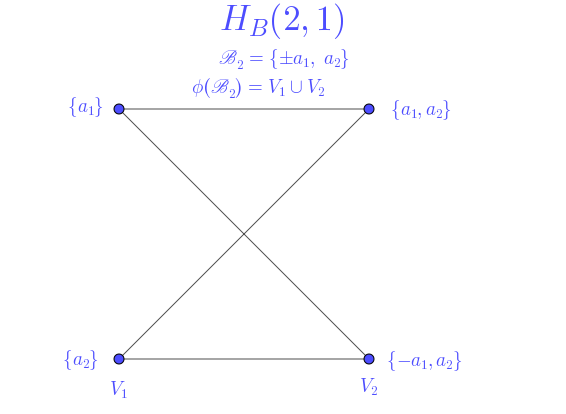}
  
  \caption{$H_B(2,1)$}
  \label{first}
  \end{figure}
\end{center}

\end{example}
\begin{example}


$H_B(n,k)$ for $n=3,\ k=2$ is illustrated in \textsc{figure} \ref{second}.

The $2$-vertex $\{1,2\}\in V_1$ is adjacent to the $1$-vertex $\{1\}$ as 
$\{1\}\subset \{1,2\}^\dag=\{|1|,|2|\}=\{1,2\}$. By the same argument, $\{1,2\}$ is adjacent to the $1$-vertex $\{2\}$. Also, $\{1,2\}$ is adjacent to $\{-1,2\}$ as $\{1,2\}\subset \{-1,2\}^\dag=\{|-1|,|2|\}=\{1,2\}$. Similar arguments explain the adjacency shown in \textsc{figure} \ref{second}.

\begin{center}
	\begin{figure}[h!]	

  \includegraphics[scale=0.4]{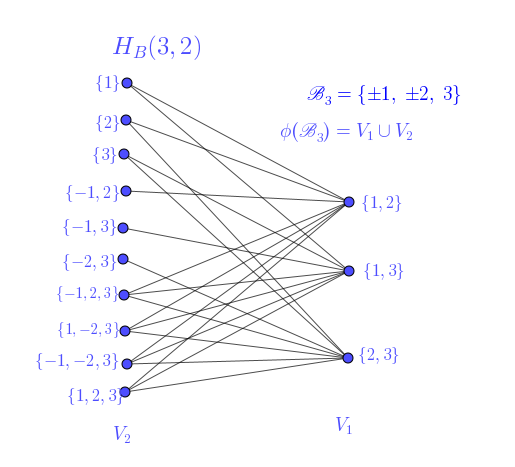}
  \caption{$H_B(3,2)$}
  \label{second}
  \end{figure}
\end{center} \end{example}

\newpage
\begin{example}\label{hbexample}

 Given the large size of $ H_{B}(4,2) $, we will present its bipartition here. According to the definition of $ H_{B}(n,k) $, a vertex $ A $ in $ V_1 $ is adjacent to a vertex $ B $ in $ V_2 $ if and only if $ A \subset B^\dag $ or $ B^\dag \subset A $.

\hspace{0.6cm}	$\mathscr{B}_4=\{ \pm 1,\; \pm 2, \;\pm 3,\; 4 \}$,
	
	  \hspace{0.7cm} $V_1=\big\{\{1,2\},\{1,3\},\{1,4\},\{2,3\},\{2,4\},\{3,4\}\big\}$,
	\begin{align*} V_2&=\big\{\{1\},\{2\},\{3\},\{4\},\{-1,2\},\{-1,3\},
	\{-1,4\}\\
	&\quad\{-2,3\},\{-2,4\},\{-3,4\},\{1,2,3\},\{-1,2,3\},\{1,-2,3\},\{-1,-2,3\},\\
	&\quad  \{1,2,4\},\{-1,2,4\},\{1,-2,4\},\{-1,-2,4\} ,\{2,3,4\},\{-2,3,4\},\\
	&\quad\{2,-3,4\},\{-2,-3,4\},\{1,3,4\},\{-1,3,4\},\{1,-3,4\},\{-1,-3,4\}\\
	&\quad \{1,2,3,4\},\{-1,2,3,4\},\{1,-2,3,4\},\{1,2,-3,4\},\{-1,-2,3,4\}\big\}\\
	&\quad \{1,-2,-3,4\},\{-1,2,-3,4\},\{-1,-2,-3,4\}\big\}.\end{align*}

 While $V_1$ contains six $2$-vertices, $V_2$ contains four $1$-vertices, six $2$-vertices, sixteen $3$-vertices, and eight $4$-vertices.
\end{example}

\section{Some parameters of bipartite Kneser B type-$k$ graphs}
\begin{theorem}
	The order of $G=H_B(n,k)$, $ |V(G)|=\dfrac{3^n-1}{2}$.
	
\end{theorem}
\begin{proof}
	Every vertex of $H_B(n,k)$ is a set formed using the elements of $\mathscr{B}_n$. Here $\mathscr{B}_n=\{ \pm a_1, \pm a_2, \pm a_3,  \dots , \pm a_{n-1}, a_n \},\text{where },  a_1<a_2\dots<a_n \;\text{and}\;-a_n\notin \mathscr{B}_n. $
	$H_B(n,k)$ has vertices as subsets of cardinality $1$ to $n$. $|V(G)|$ is the total number of sets of cardinality $1$ to $n$.\\	Let $N_i$ be the number of subsets of  $\mathscr{B}_n$ of cardinality $i$, where $1\le i\le n$.\\
	$N_1=1\binom{n}{1}=2^0\binom{n}{1}, N_2=2\binom{n}{2}, N_3=2^2\binom{n}{3},\dotsc,
	N_n=2^{n-1}\binom{n}{n}$.\\
	Thus, $|V(G)|=\sum\limits_{i=1}^n N_i=2^0\binom{n}{1}+2\binom{n}{2}+2^2\binom{n}{3}+\dotsb 2^{n-1}\binom{n}{n}= \dfrac{3^n-1}{2} $, by  binomial theorem.
\end{proof}
\begin{definition}

	A set $S\subseteq V(G)$ is an independent set of $G$ if no two vertices of $S$ are adjacent in $G$.
	The independence number of $G$, denoted by $\alpha$, is the number of vertices in a maximum independent set of $G$. $S$ is maximum if it has the maximum cardinality among all independent subsets of $G$.
\end{definition}
\begin{definition}
A set $S$ is a dominating set \cite{haynes2022domination} if for every vertex $u \in V-S$, there exists $v\in S$ such that the edge $uv \in E$. The domination number of $G$, denoted by $\gamma(G)$, is the minimum cardinality of a dominating set of $G$.
\end{definition}

\begin{theorem}
	For $G = H_B(n,k)$, the vertex independence number,
	
	$\alpha(G) = \sum\limits_{i=1}^n 2^{i-1}\binom{n}{i}-\binom{n}{k}$, for $n\ge 2$.
\end{theorem}
\begin{proof}
	By the construction of $H_B(n,k)=V_1\cup V_2$, $|V_2|$ forms a maximum independent set and 
	$|V_2|= |\phi(\mathscr{B}_n)|-\binom{n}{k}$.\\
	Therefore, $\alpha(G)=\sum\limits_{i=1}^n 2^{i-1}\binom{n}{i}-\binom{n}{k}$.
\end{proof}

\begin{corollary}
	For   $G=H_B(n,k)$, the vertex covering number, $\beta(G) = \binom{n}{k}$.
\end{corollary}
\begin{proof}
	As $\alpha(G)+\beta(G)=|G|$, we get $\beta(G)=\binom{n}{k}$.								
\end{proof}
\begin{theorem}
	The domination number, $\gamma(H_B(n,k)) = \binom{n}{k} $, for all $n\ge 2$ .
\end{theorem}
\begin{proof}

	In $H_B(n,k)$, no two vertices in $V_1$ are adjacent, and every vertex in it is adjacent to some other vertex in $V_2$. Thus, $V_1$ forms a dominating set for $H_B(n,k)$.\\$V_1$, being the smallest dominating set,we get  $\gamma(H_B(n,k)) = \binom{n}{k} $, for all $n\ge 2$.
\end{proof}
\begin{theorem}

	The size of $G=H_B(n,k)$ is $|E(G)| = \binom{n}{k} \left(\dfrac{3^k-3}{2}+2^{k-1}(3^{n-k}-1)\right)$.
	
\end{theorem}
\begin{proof}
	Let $u$ be a $k$-vertex in $V_1$.\\
	The number of $1$-vertices in $V_2$ adjacent to $u$ is $2^0\binom{k}{1}$.\\
The number of $2$-vertices in  $V_2$ adjacent to $u$ is $2^1\binom{k}{2}$.\\
The number of $3$-vertices in $V_2$ adjacent to $u$ is $2^2\binom{k}{3}$.\\
	\vdots\\
The number of  $k$-vertices in $V_2$ adjacent to $u$ is $2^{k-1}\binom{k}{k}-1$.\\
The number of $k+1$-vertices in $V_2$ adjacent to $u$ is $2^k\binom{n-k}{1}$.\\
The number of $k+2$-vertices in $V_2$ adjacent to $u$ is $2^{k+1}\binom{n-k}{2}$.\\
	\vdots\\
	The number of $n$-vertices in $V_2$ adjacent to $u$ is $2^{n-1}\binom{n-k}{n-k}$.\\
	The degree of $u$, \begin{align*} d(u)&=2^0\binom{k}{1}+2^1\binom{k}{2}+\dotsb +2^{k-1}\binom{k}{k}-1+2^k\binom{n-k}{1}\\&
	\quad +2^{k+1}\binom{n-k}{2}+\dotsb +2^{n-1}\binom {n-k}{n-k}.\\
	&= \left(\dfrac{3^k-3}{2}+2^{k-1}(3^{n-k}-1)\right).\end{align*}\\
	Since every vertex in $V_1$ is of the same degree and $d(u)$ is the maximum,
	we get,  $|E(G)| = \binom{n}{k} \left(\dfrac{3^k-3}{2}+2^{k-1}(3^{n-k}-1)\right)$.
	
\end{proof}
\begin{theorem}
	The vertex connectivity of $G=H_B(n,k)$ when $k>1$ is $\kappa(G)=1.$
\end{theorem}
\begin{proof}
	As every $k$-vertex, $v$ in $V_2$, is of degree $1$, when a $k$-vertex $u$ in $V_1$ adjacent to $v$ is removed, $v$ becomes isolated. Accordingly, we get  $\kappa(G)=1$.
\end{proof}
\begin{theorem}
	The edge connectivity of $G=H_B(n,k)$ when $k>1$is $\lambda(G)=1$.
\end{theorem}
\begin{proof}
	Let $v$ be any $k$-vertex in $V_2$. Since $d(v)=1$, the graph becomes disconnected when the only edge adjacent to it is removed. Consequently, we get  $\lambda(G)=1$.
\end{proof}

\begin{theorem}
	The Girth of $G=H_B(n,k)$ when $k>1$ is Girth(G)=4.
\end{theorem}
\begin{proof}

	Let $u$ be a $k$-vertex in $V_1$. Then, $u$ is adjacent to a $1$-vertex $x$ in $V_2$, and  $x$ is adjacent to another $k$-vertex $y$ in $V_1$. The vertex $y$ is adjacent to an $n$-vertex $v$ in $V_2$ and $v$ is adjacent to $u$. Thus, there is a cycle $u-x-y-v-u$ of length $4$. Since any cycle in a bipartite graph is of even length, eventually, we get  Girth(G)=4.

\end{proof}
\begin{theorem}
	The circuit rank of $G=H_B(n,k)$ is \\
	$\binom{n}{k}\Big(2^0\binom{k}{1}+2^1\binom{k}{2}+\dotsb +2^{k-1}\binom{k}{k}-1+2^k\binom{n-k}{1}+2^{k+1}\binom{n-k}{2}+\dotsb +2^{n-1}\binom {n-k}{n-k}\Big )\\-\sum\limits_{i=1}^n 2^{i-1}\binom{n}{i}+1$.
\end{theorem}
\begin{proof}

	The circuit rank, which is also called the cyclomatic number of a graph, is $n-m+c$, where n, m, and c are the order, size, and the number of connected components, respectively. As $c=1$ and n and m are already determined for $H_B(n,k)$, the result follows.
\end{proof}
\begin{remark}
   
    Cyclomatic  number is related to a recently defined graph invariant called the Omega invariant, which allows some combinatorial and graph theoretical properties to be calculated. The Omega invariant is an additive number defined for a given degree sequence with
    \begin{equation}\Omega(G)=\sum\limits_{i=1}^\Delta a_i(d_i-2).\end{equation}   It is shown that $\Omega(G)=2(m-n)$, and therefore it is always an even number. For further properties of the Omega invariant, see \cite{cangul18}.

\end{remark}
\begin{theorem}

	The degree of every vertex in $G=H_B(n,k)$ and the number of vertices having a specific degree are determined. The degree sequence is obtained by arranging the sequence
	$\Big\{d_{V_2}(1)^{N_{V_2}(1)}, d_{V_2}(2)^{N_{V_2}(2)}\dots ,d_{V_2}(k-1)^{N_{V_2}(k-1)},d_{V_1}(k)^{N_{V_1}(k)},\\d_{V_2}(k)^{N_{V_2}(k)}$,	$d_{V_2}(k+1)^{N_{V_2}(k+1)},\dots d_{V_2}(n)^{N_{V_2}(n)} \Big\}$of degrees with corresponding multiplicities as a monotonic non-increasing sequence.
\end{theorem}
\begin{proof}

 Let $d_{V_2}(r)$, where $r=1,2,3,\dotsc,k-1,k+1,\dotsc, n$, denote the degrees of $r$-vertices in $V_2$, and $d_{V_1}(k)$ denote the degree of any $k$-vertex in $V_1$. Let the multiplicities of degrees of any $k$-vertex in $V_1$, and $r$-vertices in $V_2$ be denoted by $ N_{V_1}(k)$ and  $ N_{V_2}(r)$ respectively.
 
 We have, 
 $d_{V_2}(1)=\binom{n-1}{k-1},\;N_{V_2}(1)=2^0\binom{n}{1} $, $d_{V_2}(2)=\binom{n-2}{k-2}, \; N_{V_2}(2)=2^1\binom{n}{2}$,
 $d_{V_2}(k-1)=\binom{n-(k-1)}{k-(k-1)}, \; N_{V_2}(k-1)=2^{k-2}\binom{n}{k-1}$, $d_{V_2}(k+1)=\binom{k+1}{k},\; N_{V_2}(k+1)=2^k\binom{n}{k+1}$, $\cdots$, $d_{V_2}(n)=\binom{n}{k}$, and  $N_{V_2}(n)=2^{n-1}\binom{n}{n}$.
 
 Let $u$ be any $k$-vertex in $V_1$. For $s=1,2,3,\dotsc,k-1$,  the number of $s$-vertices adjacent to $u$ is $2^{s-1}\binom{k}{s}$. The number of $k$-vertices adjacent to $u$ is $2^{k-1}\binom{k}{k}-1$. For $s=k+1,k+2,\dotsc,n$, the number of $s$-vertices adjacent to $u$ is 
 $2^{s-1}\binom{n-k}{t}$. Here, $t=1,2,\dotsc,n-k$. Hence, the degree of any $k$-vertex,which is denoted as  $u$,  in $V_1$ is  $d_{V_1}(k)=2^0\binom{k}{1}+2^1\binom{k}{2}+\dotsb +2^{k-1}\binom{k}{k}-1+2^k\binom{n-k}{1}+2^{k+1}\binom{n-k}{2}+\dotsb + 2^{n-1}\binom {n-k}{n-k}$. 
 
 The number of $k$-vertices in $V_1$ is $N_{V_1}(k)=\binom{n}{k}$. Every $k$-vertex in $V_2$ has degree $d_{V_2}(k)=1$. Then, the number of $k$-vertices in $V_2$ is $N_{V_2}(k)= (2^{k-1}-1)\binom{n}{k}$. Thus, the degree of every vertex in $H_B(n,k) $ and the number of vertices having a specific degree are determined. The degree sequence is obtained by arranging the sequence, 
 	$\Big\{d_{V_2}(1)^{N_{V_2}(1)}, d_{V_2}(2)^{N_{V_2}(2)}\dots ,d_{V_2}(k-1)^{N_{V_2}(k-1)},d_{V_1}(k)^{N_{V_1}(k)}, d_{V_2}(k)^{N_{V_2}(k)}$,	$d_{V_2}(k+1)^{N_{V_2}(k+1)},\dots d_{V_2}(n)^{N_{V_2}(n)} \Big\}$ of degrees with corresponding multiplicities as a monotonic non-increasing sequence.
\end{proof}

\begin{example}

	The degree sequence for $H_B(4,2) $ is obtained by arranging the sequence, $\left\{d_{V_2}(1)^{N_{V_2}(1)}, d_{V_1}(2)^{N_{V_1}(2)}, d_{V_2}(2)^{N_{V_2}(2)}, d_{V_2}(3)^{N_{V_2}(3)}, d_{V_2}(4)^{N_{V_2}(4)}\right\}=\\
 \{3^4, 19^6, 1^6, 3^{16}, 6^8\}$ of degrees with corresponding multiplicities as a monotonic, non-increasing sequence. That is, the degree sequence is $\{19^6, 6^8, 3^{16}, 3^4, 1^6\}$.
\end{example}
The following result gives the Omega invariant of $G=H_B(n,k)$:
\begin{theorem}
    The Omega invariant of $G=H_B(n,k)$ is
    
    $$\Omega(H_B(n,k))=\binom{n}{k}\sum\limits_{i=1}^n\binom{k}{i}2^{i-1}-\sum\limits_{i=1}^n\binom{n}{i}2^i.$$
\end{theorem}
\begin{proof}
    By the definition of Omega invariant, we have

    \begin{align*}
        \Omega(H_B(n,k)) &=\sum_{i=1}^n(d_i-2)N_i\\
        &=\sum_{i=1}^n\left(\binom{n-i}{k-i}-2\right)2^{i-1}\binom{n}{i}\\
        &=\binom{n}{k}\sum_{i=1}^n\binom{k}{i}2^{i-1}-\sum_{i=1}^n\binom{n}{i}2^i.
    \end{align*}
\end{proof}
In \cite{cangul18,cangul19}, it was shown that the number of closed regions of a graph is given by \begin{equation}
    r(G)=\frac{\Omega(G)}{2}+c(G).\label{eq:1}
\end{equation}
Here $c(G)$ is the number of components of $G$. Hence the number of faces of the graph $H_B(n,k)$ is given by the following result:
\begin{theorem}
    The number of faces of the graph $H_B(n,k)$ is 
    $$r(H_B(n,k))=\binom{n}{k}\sum_{i=1}^n\binom{k}{i}2^{i-2}-\sum_{i=1}^n\binom{n}{i}2^{i-1}+c .$$
\end{theorem}
\begin{proof}
    It follows by the formula of $r$ given in Eqn.(\ref{eq:1}).
\end{proof}
\begin{theorem}\label{count}

	Consider the graph $G= H_B(n,k),\:n>2,\;1<k<n\;\;\text{with}\; V(G)=V_1\cup V_2$  and  $u, v \in V(G)$, then 
	\footnotesize
	\begin{equation*}
		d(u, v) = 
		\begin{cases}
			1 & \text{if $u\in V_1$ and $v\in V_2$ are adjacent},\\
			2 & \text{if $u$ and $v$ are in $V_1$},\\
			3 & \text{if $u\in V_1$ and $v \in V_2$ are not adjacent},\\
			2,4 &  \text{if  $u$ and $v$ are in $V_2$}.
		\end{cases}
	\end{equation*}
\end{theorem}
\begin{proof}

Let $u\in V_1$ and $v\in V_2$. If $u$ and $v$ are adjacent, then $d(u,v)=1$. Suppose that $v$ is not adjacent to $u$. Since the degree of $v$ is at least $1$, it must be adjacent to some vertex $w\in V_1$. Let $x$ be an $n$-vertex in $V_2$. As $x$ is a common neighbour of $u$ and $w$, $u-x-w-v$ is the shortest path from $u$ to $v$, and hence $d(u,v)=3$.
Let $u,v\in V_1$. As any $n$-vertex $x$ is a common neighbour of $u$ and $v$, $u-x-v$ is the shortest path from $u$ to $v$, and hence $d(u,v)=2$.
Let $u,v\in V_2$. Then, $d(u,v)=2$ in one of the following three cases.
\begin{description}

\item[Case 1]There exists $x\in V_1$ such that $x$ is a superset of both $u^\dag$ and $v^\dag$.\\\item[Case 2] There exists $y\in V_1$ such that $y$ is a subset of both $u^\dag$ and $v^\dag$ .\\\item[Case 3]There exists $w\in V_1$ such that $w$ is a subset of $u^\dag$ and a\\\hspace*{1cm} super set of $v^\dag$ or $w$ is a subset of $v^\dag$ and a superset of $u^\dag$ .
\end{description}

In other words, $d(u, v)= 2$ if either $|u^\dag\cup v^\dag|\le k$ or $|u^\dag\cap v^\dag|\ge k$ or $|u^\dag\cap v^\dag|=|u^\dag|\le k$. If none of these three conditions are satisfied, then $d(u,v)\ne 2$. Choose $x,y\in V_1$ such that $d(x,u)=1$ and $d(y,v)=1$. Let $w$ be any $n$-vertex in $V_2$. Then $w$ is a common neighbour of $x$ and $y$. Thus, we get the shortest path $u-x-w-y-v$ of length $4$ and hence $d(u,v)=4$.
\end{proof}

The following proposition and corollary are immediate consequences of the lemma.
\begin{corollary}
    
\label{P37}
	
	If $n> 2$ and $1<k<n$, for the graph $H_{B}(n, k)$, the eccentricity is given by 
	\begin{equation*}
	e(v)= 
	\begin{cases}
	3 & \text{if $v\in V_1$,}\\
	2 & \text{if $v\in V_2$ is an n-vertex },\\
	4  & \text{if $v\in V_2$ is an r-vertex ,$1\le r<n$. } 	
	\end{cases}
	\end{equation*}
	
\end{corollary}
\begin{proof}
    
Let $v\in V_1$. For any $u\in V_1$, $d(v,u)=2$. If $w\in V_2$ such that it is adjacent to $v$, then $d(v,w)=2$. If $w$ is not adjacent to $v$, then $d(v,w)=3$. Thus, $e(v)=3$.
Let $v\in V_2$ be an $n$-vertex. Since $v$ is adjacent to all vertices in $V_1$ and $d(v,u)=2$ for any $u\in V_2$, we get $e(v)=2$.
Let $v\in V_2$ be an $r$-vertex, $1\le r < n$.
The distance from $v$ to any vertex in $V_1$ is either $1$ or $3$, and the distance from $v$ to any vertex in $V_2$ is either $2$ or $3$.

\end{proof}
\begin{corollary}
	The diameter of $G=H_B(n,k),n>2,1<k<n $ is $\emph{diam}(G)=4$ and radius is $\emph{rad}(G)=2$.
\end{corollary}
\begin{proof}
    $\text{diam}(G)=\text{max}\{e(v): v\in V\}=\text{max}\{2,3,4\}=4$.
    
  \hspace{1cm}  $\text{rad}(G)=\text{min}\{e(v): v\in V\}=\text{min}\{2,3,4\}=2$.
\end{proof}
\begin{corollary}

	For $G=H_B(n,k)), \:n>2,\;1<k<n $, 
	\begin{enumerate}
		\item $\emph{Periphery},P(G)=\{v\in V_2 | v \;\text{is an    r-vertex},1\le r<n \}.$
		\item  \emph{Center}, $C(G)=\{v\in V_2 |\text{ v is an n-vertex} \}.$
		\item \emph{ Median}, M(G)= C(G).
	\end{enumerate}
\end{corollary}
\begin{proof}

 As the eccentricity of any $r$-vertex $v$, where $1\le r<n$ is $e(v)=4=\text{diam}(G)$, we get, $\text{Periphery}, P(G)=\{v\in V_2|   v \;\text{is an r-vertex},1\le r<n \}$ As the eccentricity of any $n$-vertex $v$ is $e(v)=2=\text{rad}(G)$, we get, $\text{Center}, C(G)=\{v\in V_2|   v \;\text{is an n-vertex}\}$.

 For finding the median($G$), we find the status of vertices in $V$. We have the status of any vertex $v\in G$,  $s(v)=\sum\limits_ {u\in V(G)} d(v,u)$. First, we find the status of any $n$-vertex $v$ in $V_2$. The sum of the distances from $v$ to $\binom{n}{k}$ vertices in $V_1$ is $\binom{n}{k}$.  The sum of the distance from $v$ to $\frac{3^n-1}{2}-\binom{n}{k}-2^{n-1}$, $r$-vertices, where $1\le r <n$ in $V_2$  is  $2\left(\frac{3^n-1}{2}-\binom{n}{k}-2^{n-1}\right)$. Therefore, status of $v$ is $2\left(\frac{3^n-1}{2}-\binom{n}{k}-2^{n-1}\right)+\binom{n}{k}$. There are vertices at distances $1, 2, $ and $3$ from vertices in $V_1$. There are vertices at distances $1, 2, 3$, and $4$ from  $r$-vertices in $V_2 $. This leads to the conclusion that the status of any $n$-vertex is minimum compared to other vertices in $V$. Thus, Median, $M(G)$=$C(G)=\{v\in V_2|\text{ v is an n-vertex} \}.$

\end{proof}
\begin{remark}
	
	Let $G=H_B(n,k)$ with bipartition, $V(G)=V_1\cup V_2$. We denote the set of all  pairs (unordered) of vertices of $V(G)$ by
	 $S=\{\{u,v\}\mid u,v\in V(G)\}$. Then, $S$ contains vertex pairs at distances $1, 2, 3, $ and $4$. The subsets of $S$ are of the form $S\big(V(G),h\big)=\{\{u,v\}\mid  d(u,v)=h, \;\; 1\le h\le 4\}$. Then, $S=\bigcup_{h=1}^4 S\big(V(G),h\big)$.  Let $d^{(h)}(u, v)$ be  the cardinality of $S\big(V(G),h\big)$ for $h=1,2,3,$ and $4$. We denote by $S(V_1,2)$ and $S(V_2,2)$, respectively,  the sets of vertex pairs of $V_1$ and $V_2$ at distance $2$.
	  We have, $S\big(V(G),2\big)=S(V_1,2)\cup S(V_2,2)$. Let $d^{(2)}_{V_1}(u,v)$ and $d^{(2)}_{V_2}(u,v)$, respectively, denote the cardinalities of $S(V_1,2)$ and $S(V_2,2)$.
	Consequently, we get,  $d^{(2)}(u,v)=d^{(2)}_{V_1}(u,v) +d^{(2)}_{V_2}(u,v)$. As vertex pairs at distance $4$ exists only in $V_2$, we denote the set containing them as $S(V_2,4)$. We have, $S\big(V(G),4\big)=S(V_2,4)$. Let  $d^{(4)}_{V_2}(u,v)$ be the cardinality of $S(V_2,4)$. Then, 
	$d^{(4)}(u, v)=d^{(4)}_{V_2}(u,v)$.
\end{remark}

 The cardinalities of $S\big(V(G),1\big)$ and  $S\big(V(G),3\big)$, the total number of vertex pairs at distance $2$ from $V_2$ denoted by $d^{(2)}_{V_2}(u,v) $, and the total number of vertex pairs $\{u,v\}$ at distances of $2$ and $4$, where $u$ and $v$ are from $V_2$, are determined in the next theorem
\begin{theorem}\label{prop18}
	Consider  the bipartite Kneser B type-k graph $G=H_B(n,k)), \:n>2,\;1<k<n $. For $u,v\in V$, 
	\begin{enumerate}
\item 	$d^{(1)}(u,v)=\binom{n}{k} \left(\dfrac{3^k-3}{2}+2^{k-1}(3^{n-k}-1)\right)$.

\item  $d^{(3)}(u,v)=\binom{n}{k}\left(\frac{3^n-1}{2}-\binom{n}{k}-\left(\frac{3^k-3}{2}\right)-2^{k-1}(3^{n-k}-1)\right)$.
\item  $d^{(2)}_{V_1}(u,v)=\binom{\binom{n}{k}}{2}$.
\item $d_{V_2}^{(2)}(u,v) + d_{V_2}^{(4)}(u,v)=\dbinom{\frac{3^n-1}{2}-\tbinom{n}{k}}{2}$.\label{problem}
\end{enumerate}

\end{theorem}

\begin{proof}
	$d(u,v)=1$ when $u\in V_1$ and $v\in V_2$ are adjacent. Thus, $d^{(1)}(u,v)=|E(G)|=\binom{n}{k} \left(\dfrac{3^k-3}{2}+2^{k-1}(3^{n-k}-1)\right)$. 
	$d(u,v)=2$ when $u$ and $v$ are in $V_1$. As there are $\binom{n}{k}$ $k$-vertices in $V_1$, $d^{(2)}_{V_1}(u,v)=\binom{\binom{n}{k}}{2}$. $d(u,v)=3$ when $u\in V_1$ and $v\in V_2$ are not adjacent.
	
	The total number of unordered pairs of vertices such that one vertex  is from $V_1$ and the other from $V_2$ is $\binom{n}{k}(|V(G)|-\binom {n} {k} )$. 	Thus, $d^{(3)}(u,v)= \binom{n}{k}(|V(G)|-\binom {n} {k} )-|E(G)|=\binom{n}{k}\left(\frac{3^n-1}{2}-\binom{n}{k}-\left(\frac{3^k-3}{2}\right)-2^{k-1}(3^{n-k}-1)\right). $

	Given that \( d_{V_2}^{(2)}(u, v) \) and \( d_{V_2}^{(4)}(u, v) \) are the counts of pairs at distances 2 and 4 respectively, it follows from theorem \ref{count} that
	
	\begin{align*}d_{V_2}^{(2)}(u,v) + d_{V_2}^{(4)}(u,v)&=\binom{|V(G)|}{2}-\left(|E(G)|+\binom{n}{k}\left(|V(G)|-\binom{n}{k}\right)-|E(G)|\right)\\&=\dbinom{\frac{3^n-1}{2}-\tbinom{n}{k}}{2}\end{align*}.

\end{proof}

 $d_{V_2}^{(4)}(u,v)$ is computed in the following theorem. As any $i$-vertex $u$ and $j$-vertex $v$ in $V_2$ can have both positive and negative components, whenever we say common elements in an $i$ vertex and a $j$-vertex, we mean $|u^\dag \cap v^\dag|$.
\begin{theorem}\label{prop19}

    Let $P_{i,j}$ be the number of unordered pairs of $i$-vertices and $j$-vertices of $V_2$ that are at distance $4$. 
    
    For $u, v\in V_2$ in $H_B(n,k)$, $d_{V_2}^{(4)}(u,v)=\sum\limits_{\substack{i,j\in \{1,2,\dotsc n-1\}\\i\leq j\\k+1\leq i+j \leq n+k-1}} P_{i,j}$.
    
    Here,
$P_{i,j}=\begin{cases}\sum\limits_t\binom{n}{i}2^{i-1}2^{j-1}\binom{n-i}{j-t}\binom{i}{t}\quad \text{for $i\neq k$ and $j\neq k $}\\

\sum\limits_t\binom{n}{k}(2^{k-1}-1)2^{j-1}\binom{n-k}{j-t}\binom{k}{t}\quad\text{for}\ i= k\; \text{and}\;j\neq k\\

\sum\limits_t\binom{n}{i}2^{i-1}(2^{k-1}-1)\binom{n-i}{k-t}\binom{i}{t}\quad\text{for}\ i\neq k\; \text{and}\;j= k \\
\frac{1}2{}\left(\sum\limits_t\binom{n}{k}(2^{k-1}-1)^2\binom{n-k}{k-t}\binom{k}{t}\right)\quad\text{for}\ i=j=k\\

\frac{1}{2}\left(\sum\limits_t\binom{n}{i}2^{2(i-1)}\binom{n-i}{i-t}\binom{i}{t}\right) for\; i=j\; \text{ and } j\ne k

\end{cases}$

Here, $|u^\dag\cap v^\dag |=t$ and $t$ is a non-negative integer such that $i+j-n\le t< i+j-k$ and $t< \text{min}\{i,k\}$

\end{theorem}

\begin{proof}
    
Choose an $i$-vertex $u=\{x_1,x_2,\dotsc,x_i\}$ and a $j$-vertex $v=\{y_1,y_2,\dotsc,y_j\}$ in $V_2$ such that $d(u,v)=4$. Here $i,j\in \{1,2,\dotsc n-1\},i\leq j,\text{and}\:k+1\leq i+j \leq n+k-1$.By the construction of $H_B(n,k)$, the  vertices at distance $4$ satisfy the conditions: $|u^\dag\cap v^\dag|=t$, $t\ge 0$, $i+j-n\le t< i+j-k$, and $t< \text{min}\{i,k\}$. Corresponding to an $i$-element subset of $\mathscr{B}_n^+=\{ a_1,  a_2,  a_3, \dots,  a_{n-1}, a_n \}$, $2^{i-1}$, $i$-element subsets or $i$-vertices are seen in $\phi(\mathscr{B}_n)$. Similarly, corresponding to a $j$-element subset of $\mathscr{B}_n^+$, $2^{j-1}$, $j$-vertices are there in $\phi(\mathscr{B}_n)$. $P_{i,j}$ is calculated in various cases.
\begin{description}
    
 \item[Case 1] For $i,j$ such that $i< j$ and $i\ne k$ and $j\ne k$.

 There are $2^{j-1}$, $j$-vertices at distance $4$ to $u$. As there are $2^{i-1}$ vertices corresponding to $u$, the total number of $4$ pairs between $u$ and $v$ and their corresponding vertices is $2^{i-1}2^{j-1}$. As $|u^\dag\cap v^\dag|=t$, the remaining $j-t$ elements in any other $j$-vertex at distance $4$ can be selected from $n-i$ elements of $\mathscr{B}_n^+$ in $\binom{n-i}{j-t}$ ways. Also, $t$ elements can be selected from $i$-vertex in $\binom{i}{t}$ways.   The total number of $i$ element subsets of $\mathscr{B}_n^+$ is $\binom{n}{i}$. Using the restrictions on $t,i \;\; \text{and} \;\;j $, we get $P_{i,j}=\sum\limits_t\binom{n}{i}2^{i-1}2^{j-1}\binom{n-i}{j-t}\binom{i}{t}$.

 \item[Case 2] For $i,j$ such that $i< j$, $i=k$ and $j\ne k$. 
 
 Of the $2^{k-1}$, $k$-vertices corresponding to $\{x_1,x_2,\dotsc,x_k\}$ in $\phi(\mathscr{B}_n)$, $(2^{k-1}-1)$are in $V_2$ and $1$ in $V_1$. Therefore, $P_{i,j}=\sum\limits_t\binom{n}{i}(2^{k-1}-1)2^{j-1}\binom{n-k}{j-t}\binom{k}{t}$.\item[Case 3] For $i,j$ such that $i< j$, $i\ne k$ and $j=k$. 
 
 Of the $2^{k-1}$, $k$-vertices corresponding to $\{y_1,y_2,\dotsc,y_k\}$ in $\phi(\mathscr{B}_n)$, 
 $(2^{k-1}-1)$are in $V_2$ and $1$ in $V_1$. Therefore, $P_{i,j}=\sum\limits_t\binom{n}{i}2^{i-1}(2^{k-1}-1)\binom{n-i}{k-t}\binom{i}{t}$.
 \item[case 4] For $i,j$ such that $i=j=k$.
 
  Using similar arguments, we conclude that $P_{i,j}=\frac{1}{2}\left(\sum\limits_t\binom{n}{k}(2^{k-1}-1)^2\binom{n-k}{k-t}\binom{k}{t}\right)$  \item[Case 5] For $i,j$ such that $i=j$ and $j\ne  k $. 
  
  Here,  $P_{i,j}=\frac{1}{2}\left(\sum\limits_t\binom{n}{i}2^{2(i-1)}\binom{n-i}{i-t}\binom{i}{t}\right) .$
 
   \end{description}

For $i,j\in \{1,2,\dotsc n-1\},i\leq j,\text{and}\:k+1\leq i+j \leq n+k-1$, the total number of unordered pairs of vertices from  $V_2$ such that $d(u,v)=4$ is 
\begin{align*}
    d^{(4)}_{V_2}(u,v) =&\sum P_{i,j}\\
    =&P_{1,k}+P_{1,k+1}+\dotsb P_{1,n-1}+\\
    & P_{2,k-1}+P_{2,k}+\dotsb P_{2,n-1}+\\
    &\dotsb\dotsb\dotsb\dotsb\dotsb\dotsb\dotsb\dotsb+\\
    & P_{k,k}+\dotsb P_{k,n-1}+P_{k+1,k+1}+\dotsb P_{k+1,n-2}+\\
    &\dotsb\dotsb\dotsb\dotsb\dotsb\dotsb\dotsb\dotsb+\\
    & P_{\frac{n+k-1}{2},\frac{n+k-1}{2}}.  \qquad    (\text {When} \;\;  n+k-1  
   \;\;\text{is even})
\end{align*}
\begin{align*}
    d^{(4)}_{V_2}(u,v) =&\sum P_{i,j}\\
    =&P_{1,k}+P_{1,k+1}+\dotsb P_{1,n-1}+\\
    & P_{2,k-1}+P_{2,k}+\dotsb P_{2,n-1}+\\
    &\dotsb\dotsb\dotsb\dotsb\dotsb\dotsb\dotsb\dotsb+\\
    & P_{k,k}+\dotsb P_{k,n-1}+P_{k+1,k+1}+\dotsb P_{k+1,n-2}+\\
    &\dotsb\dotsb\dotsb\dotsb\dotsb\dotsb\dotsb\dotsb+\\
    & P_{\frac{n+k-2}{2},\frac{n+k-2}{2}} +  P_{\frac{n+k-2}{2},\frac{n+k}{2}}.\qquad    (\text {When} \;\;  n+k-1  
   \;\;\text{is odd})
\end{align*}

\end{proof}
\begin{remark}
	
	From theorem \ref{prop18}, we have got the cardinalities of $S\big(V(G),1\big)$ and  $S\big(V(G),3\big)$. Using theorems \ref{prop18} and \ref{prop19}, the cardinalites of $S\big(V(G),2\big)$ and $S\big(V(G),4\big)$ are obtained as $d^{(2)}(u,v)=d^{(2)}_{V_1}(u,v) +d^{(2)}_{V_2}(u,v)$ and  $d^{(4)}(u,v)=d^{(4)}_{V_2}(u,v)$.
\end{remark}

\begin{example}
    
    Consider the partite sets $V_1$ and $V_2$ of $H_B(4,2)$ as given in example \ref{hbexample}.
    
    	\begin{align*}d^{(1)}(u,v)&=\binom{n}{k} \left(\dfrac{3^k-3}{2}+2^{k-1}(3^{n-k}-1)\right).\\
    	&=\binom{4}{2} \left(\dfrac{3^2-3}{2}+2^{2-1}(3^{4-2}-1)\right).\\
    	&=114.\end{align*}
    	 \hspace*{1.2cm}$d^{(2)}_{V_1}(u,v)\;=\binom{\binom{n}{k}}{2}=\binom{\binom{4}{2}}{2}=15.$
    	 
    	 \begin{align*}d^{(3)}(u,v)&=\binom{n}{k}\left(\frac{3^n-1}{2}-\binom{n}{k}-\left(\frac{3^k-3}{2}\right)-2^{k-1}(3^{n-k}-1)\right).\\
    	 &=\binom{4}{2}\left(\frac{3^4-1}{2}-\binom{4}{2}-\left(\frac{3^2-3}{2}\right)-2^{2-1}(3^{4-2}-1)\right)\\
    	 &=90.\end{align*}
    	
     Then, $d^{(4)}(u,v)=d^{(4)}_{V_2}(u,v)=\sum\limits_{\substack{i,j\in \{1,2,3\}\\i\leq j\\3\leq i+j \leq 5}} P_{i,j}=P_{1,2}+P_{1,3}+P_{2,2}+P_{2,3}$.

    \vspace{1cm}
   For $t=0$,  $P_{1,2}=\binom{4}{1}2^{1-1}(2^{2-1}-1)\binom{4-1}{2-0}\binom{1}{0}=12$.

    For $t=0$, $P_{1,3}=\binom{4}{1}2^{1-1}2^{3-1}\binom{4-1}{3-0}\binom{1}{0}=16$.   
    
   For $t=0,1$, $P_{2,2}=\frac{1}{2}\left(\binom{4}{2}(2^{2-1}-1)^2)\binom{4-2}{2-0}\binom{2}{0}+\binom{4}{2}(2^{2-1}-1)^2\binom{4-2}{2-1}\binom{2}{1}\right)=15$.

      For $t=1$,  $P_{2,3}=\binom{4}{2}(2^{2-1}-1)2^{3-1}\binom{4-2}{3-1}\binom{2}{1}=48$.
      
      Therefore, $d^{(4)}(u,v)=d^{(4)}_{V_2}(u,v)=12+16+15+48=91$.

      Using the  identity (\ref{problem}) in theorem (\ref{prop18}),  we get

      $d_{V_2}^{(2)}(u,v) + d_{V_2}^{(4)}(u,v)=\binom{\frac{3^n-1}{2}-\binom{n}{k}}{2}$.
      
      Therefore, $d_{V_2}^{(2)}(u,v)=\binom{\frac{3^n-1}{2}-\binom{n}{k}}{2}-d_{V_2}^{(4)}(u,v)
      =561-91=470 $

      $d^{(2)}(u,v)=d^{(2)}_{V_1}(u,v) +d^{(2)}_{V_2}(u,v)=15+470=485$.

      \begin{table}[ht]
      	\renewcommand\arraystretch{1.5}
      	
      	\centering
      	\caption{A table showing the values of $d^{(h)}(u,v)$, $1\le h\le 4$ and $H_B(n,k)$ for some  values of $n$ and $k$.}
      	\begin{tabular}[t]{lrrrr}
      		\hline

            &$d^{(1)}(u,v)$&$d^{(2)}(u,v)$&$d^{(3)}(u,v)$&$d^{(4)}(u,v)$\\
      		\hline
      		$H_B(4,2)$&114&485&90&91\\
      		$H_B(4,3)$&80&486&64&150\\
      		$H_B(5,2)$&550&5275&560&875\\
      	
      		$H_B(5,3)$&440&4125&670&2025\\
      		$H_B(5,4)$&275&4715&305&1965\\
      		$H_B(6,2)$&2445&54050&2790&6781\\
      		\hline
      	\end{tabular}
      \end{table}%

\end{example}

\section{Conclusion}

In this paper, we have determined various invariants of $H_B(n,k)$. There is scope for applications in various disciplines. As the degree sequence of $H_B(n,k)$ and the distance between any pair of vertices in  $H_B(n,k)$ are determined, hundreds of molecular descriptors can be calculated. Distance properties also lead to the determination of various types of metric dimensions. Centrality measures such as degree centrality, Closeness centrality, betweenness centrality and eigen vector centrality can also be obtained.

\subsection*{Conflict of Interest}

The authors hereby declare that there is no potential conflict of interest.
\subsection*{Acknowledgement}

The first author is a doctoral fellow in mathematics at University College, Thiruvananthapuram. This research has been promoted and supported by  the University of Kerala.

\bibliographystyle{amsplain}
\bibliography{bibliohbnkpart1}
\end{document}